\documentclass[preprint,12pt,3p]{elsarticle}
\usepackage{amsmath, amsthm, amssymb}           % Equation environment
\usepackage{blkarray}
\usepackage{mathrsfs}                           % \mathscr
\usepackage{physics}
\usepackage{enumitem}
\usepackage{caption}
\usepackage{subcaption}
\usepackage{float}
\usepackage{tikz-cd}

\newtheorem{theorem}{Theorem}[section]          % Theorem environment
\newtheorem{lemma}[theorem]{Lemma}              % Theorem environment
\newtheorem{corollary}[theorem]{Corollary}      % Theorem environment
\newtheorem{observation}[theorem]{Observation}

\theoremstyle{definition}
\newtheorem{definition}[theorem]{Definition} 	% Definition environment
\newtheorem{remark}[theorem]{Remark}
 
\newtheorem{example}[theorem]{Example}

\newcommand{\bb}[1]{\mathbb{#1}}								% Blackboard boldface for R, C, N, Q
\newcommand{\coni}[1]{\operatorname{coni}\left( #1 \right)}     % conical hull
\newcommand{\conv}[1]{\operatorname{conv}\left( #1 \right)}     % conical hull{amssymb}

\journal{Linear Algebra and its Applications}

\begin{document}

%-------------------------------------------------------------------
\begin{frontmatter}
\title{Realizing Sule\u{\i}manova spectra via permutative matrices, II}
%\tnotetext[label0]{This is only an example}                     % Title note test

% Pietro contact
\author[uwb]{Pietro Paparella\corref{cor1}}
\address[uwb]{Division of Engineering and Mathematics, University of Washington Bothell, Bothell, WA 98011-8246, USA}
\cortext[cor1]{Corresponding author.}
%\fntext[label4]{Small city}                         % Footnote text!
\ead{pietrop@uw.edu}
\ead[url]{http://faculty.washington.edu/pietrop/}

% Amber contact
\author[ua]{Amber R.~Thrall\fnref{marygatesthanks}}
\address[ua]{Department of Mathematics, University of Arizona, Tucson, AZ 85721-0089 USA}
\ead{arthrall@math.arizona.edu}
\ead[url]{https://amber.thrall.me/}
\fntext[marygatesthanks]{Supported by a University of Washington Mary Gates Research Scholarship.}

\begin{abstract}
In this work, the real nonnegative inverse eigenvalue problem is solved for a particular class of permutative matrix. The necessary and sufficient condition there is also shown to be sufficient for the symmetric nonnegative inverse eigenvalue problem. A result due to Johnson and Paparella [MR3452738; Linear Algebra Appl.~493 (2016), 281--300] is extended to include normalized lists that satisfy the new sufficient condition.
\end{abstract}


\end{frontmatter}

%---------------------------------
\section{Introduction}
%---------------------------------

The longstanding \emph{real nonnegative inverse eigenvalue problem} (RNIEP) is to determine necessary and sufficient conditions on a multiset (herein, \emph{list}) of real numbers $\Lambda = \{ \lambda_1, \dots, \lambda_n \}$ such that $\Lambda$ is the spectrum of an $n$-by-$n$ nonnegative matrix (entrywise). 

If $A$ is an $n$-by-$n$ nonnegative matrix with spectrum $\Lambda$, then $\Lambda$ is said to be \emph{realizable} and the matrix $A$ is called a \emph{realizing matrix} for $\Lambda$. If the realizing matrix is required to be symmetric, then the problem is known as the \emph{symmetric nonnegative inverse eigenvalue problem} (SNIEP). Both problems are unsolved when $n \ge 5$.

It is well-known that if $\Lambda$ is realizable, then  
\begin{equation}
s_k (\Lambda) := \sum_{i=1}^n \lambda_i^k = \trace(A^k) \geq 0,~\forall~k \in \bb{N}  \label{cond1}
\end{equation}
and 
\begin{equation}
\rho(\Lambda) := \max_{1\leq i \leq n} \{ |\lambda_i| \} \in \Lambda.           \label{cond2}
\end{equation}
Condition \eqref{cond2} follows from the Perron-Frobenius theorem. For additional background and results on the RNIEP and the general \emph{nonnegative inverse eigenvalue problem (NIEP)}, there are several surveys \cite{egleston2004,niepsurvey} and a monograph \cite{minc1988}.

A list of real numbers $\Lambda = \{ \lambda_1, \dots, \lambda_n \}$ is called a \emph{Sule\u{\i}manova spectrum} if $s_1(\Lambda) \geq 0$ and $\Lambda$ contains exactly one positive element. In the seminal work on the NIEP, Sule\u{\i}manova \cite{suleimanova1949} announced that every such spectrum is realizable. Friedland \cite{friedland1978} and Perfect \cite{perfect1953} proved Sule\u{\i}manova's result via companion matrices (for other proofs, see references in \cite{friedland1978}). More recently, Paparella \cite{paparella2016} gave a constructive proof that every Sule\u{\i}manova spectrum is realizable by utilizing \emph{permutative matrices}, which are defined in the sequel. 

Fiedler \cite{fiedler1974} showed that every Sule\u{\i}ma\-nova spectrum is \emph{symmetrically realizable}. Since every symmetric matrix is diagonalizable, it follows that every Sule\u{\i}manova spectrum is diagonalizably realizable, and, hence, a solution to the \emph{diagonzalizable real nonnegative inverse eigenvalue problem} (DRNIEP). Johnson et al.~\cite{johnsonlaffeyloewy1996} showed that the RNIEP and SNIEP are distinct; Laffey \cite{laff1995} showed that not every realizable list is diagonalizably realizable;  and, more recently, Cronin and Laffey \cite{cl2017} showed that the SNIEP and the DRNIEP are distinct. Thus, with a slight abuse of notation, we have
\[ \text{SNIEP}_n \subset \text{DRNIEP}_n \subset \text{RNIEP}_n. \]

It is natural to ask whether every Sule\u{\i}ma\-nova spectrum is realizable by a nonnegative asymmetric matrix. A positive answer is given by Soto and Ccapa \cite{sotoccapa2008}, however it will be shown in the sequel that permutative matrices provide a simple proof of this fact. In addition, we solve the RNIEP for a particular class of permutative matrix and derive a novel sufficient condition for the SNIEP. We find the position of this sufficient condition with respect to other well-known conditions for sufficiency \cite{mps2007,mps2017}. A result due to Johnson and Paparella \cite[Theorem 6.3]{johnsonpaparella2016} is extended to include normalized lists that satisfy the new sufficient condition.

%---------------------------------
\section{Notation \& Background}
%---------------------------------

The following notation is adopted throughout this work. 

For $n \in \mathbb{N}$, we let $\langle n\rangle$ denote the set $\{1,\dots,n\}$. We let $\mathbb{F}^n$ denote the set of all column vectors with entries from a field $\mathbb{F}$ and ${M}_n(\mathbb{F})$ denote the set of all square $n$-by-$n$ matrices with entries from a field $\mathbb{F}$. As is customary, if $A \in M_n (\mathbb{F})$, then $a_{ij}$ denotes the $(i,j)$-entry of $A$; if $x \in \mathbb{F}^n$, then $x_i$ denotes the $i\textsuperscript{th}$ entry of $x$. For $x \in \mathbb{F}^n$, we let $\text{diag}(x)$ denote the diagonal matrix such that $d_{ii} = x_i$. The spectrum of $A \in M_n(\mathbb{F})$ is denoted by $\sigma(A)$.

For the following, the size of each object will be clear from the context in which it appears: $e$ denotes the all-ones column vector; $J$ denotes the all-ones matrix, i.e., $J:=ee^\top$; $I$ denotes the identity matrix; and $e_i$ denotes the $i$th column of $I$ (note that this is an exception to the convention stated above). 

%---------------------------------
\begin{definition}
\label{defn:suleimanova}
A list of real numbers $\Lambda=\{\lambda_1,\dots,\lambda_n\}$ is called a \textit{Sule\u{\i}manova spectrum} if \(s_1(\Lambda)\ge 0\) and $\Lambda$ contains exactly one positive element.
\end{definition}

Denote by \(S_n\) the symmetric group of order \(n\). For \( \pi \in S_n\) and \(x \in \mathbb{F}^n\), let \( \pi(x)^\top := \begin{bmatrix} x_{\pi(1)} & \cdots & x_{\pi(n)} \end{bmatrix}^\top \).  

%---------------------------------
\begin{definition}
\label{defn:permutative}
A matrix $P \in {M}_n(\mathbb{C})$ is called \emph{permutative} or a \emph{permutative matrix} if
\begin{equation*}
P = \begin{bmatrix}
\pi_1(x)^\top	\\
\vdots          	\\
\pi_n(x)^\top
\end{bmatrix},
\end{equation*}
where \(\pi_1,\dots,\pi_n \in S_n\) and \(x \in \mathbb{C}^n\).
\end{definition}

%---------------------------------
\begin{remark}
Definition \ref{defn:permutative} differs from the definition of permutative matrix given in \cite{paparella2016}, which involves permutation matrices. Since permutation matrices are permutative, it follows that the definition given here is exactly what is desired.
\end{remark}

In \cite{paparella2016}, permutative matrices were utilized to give a simple, constructive proof of the following result.  

%---------------------------------
\begin{theorem}
[Sule\u{\i}manova \cite{suleimanova1949}]
\label{thm:paparella2016}
    Every Sule\u{\i}manova spectrum is realizable.
\end{theorem}

We summarize the relevant results from \cite{paparella2016}. For $x \in \mathbb{C}^n$, if 
\begin{equation}
\label{eq:perm_matrix}
P_x :=
\begin{blockarray}{ccccccc}
& \mbox{\scriptsize1} & \mbox{\scriptsize2} & \mbox{\scriptsize\dots} & \mbox{\scriptsize $i$} & \mbox{\scriptsize\dots} & \mbox{\scriptsize $n$} \\
\begin{block}{c[ cccccc ]}
    \mbox{\scriptsize1} & x_1 & x_2 & \dots & x_i & \dots & x_n \\
    \mbox{\scriptsize2} & x_2 & x_1 & \dots & x_i & \dots & x_n \\
    \mbox{\scriptsize\vdots} & \vdots & \vdots & \ddots & \vdots & & \vdots \\
    \mbox{\scriptsize $i$} & x_i & x_2 & \dots & x_1 & \dots & x_n \\
    \mbox{\scriptsize\vdots} & \vdots & \vdots & & \vdots & \ddots & \vdots \\
    \mbox{\scriptsize $n$} & x_n & x_2 & \dots & x_i & \dots & x_1 \\
\end{block}
\end{blockarray}
,
\end{equation}
then 
\[ \sigma(P_x) = \left\{ \sum_{i=1}^n x_i, x_1 - x_2, \dots,x_1 - x_n \right\}. \] 

Furthermore, if $\Lambda = \{ \lambda_1,\lambda_2, \dots, \lambda_n \}$ is a list of complex numbers and 
\begin{equation}
    x := \frac{1}{n}\begin{bmatrix}
        1 & e^\top \\
        e & J-nI
    \end{bmatrix}
    \begin{bmatrix}
    \lambda_1 \\
    \vdots \\
    \lambda_n
    \end{bmatrix}
    =
    \frac{1}{n}\begin{bmatrix}
        s_1(\Lambda) \\
        s_1(\Lambda) - n\lambda_2 \\
        \vdots \\
        s_1(\Lambda)-n\lambda_n
    \end{bmatrix},                      \label{novsuffcond}
\end{equation}
then $P_x$ has spectrum $\Lambda$.

Let $\Sigma := \sum_{k=1}^n x_k$. Since every row of $P_x$ sums to $\Sigma$, it follows that $(\Sigma,e)$ is a right-eigenpair for $P_x$. Furthermore, if $\delta_i := x_1 - x_i \ne \Sigma$, then $ (\delta_i,v_i)$ is a right-eigenpair for $P_x$, where 
\begin{equation}
\label{eq:vvec}
v_i :=
\begin{blockarray}{cc}
\begin{block}{c[c]}
\mbox{\scriptsize1} & x_i \\
\mbox{\scriptsize\vdots} & \vdots \\
\mbox{\scriptsize $i-1$} & x_i \\
\mbox{\scriptsize $i$} & x_1-\Sigma \\
\mbox{\scriptsize $i+1$} & x_i \\
\mbox{\scriptsize\vdots} & \vdots \\
\mbox{\scriptsize $n$} & x_i \\
\end{block}
\end{blockarray}
,~i \ne 1.    
\end{equation}

%-----------------------------------------------------------------
\section{Diagonalizable Realizibility of Sule\u{\i}manova Spectra}
%-----------------------------------------------------------------

%---------------------------------
\begin{lemma}
The matrix formed from the eigenvectors of \eqref{eq:perm_matrix} given by
\[ S = \begin{bmatrix}e&v_2&v_3&\dots&v_n\end{bmatrix}, \]
where $n\ge2$, is invertible whenever $x \ge 0$ and $x_1 \ne \Sigma$.
\end{lemma}

\begin{proof}
Following the properties of the determinant, 
\begin{align*}
|S| &= 
\begin{blockarray}{ccccccc}
& \mbox{\scriptsize1} & \mbox{\scriptsize2} & \mbox{\scriptsize\dots} & \mbox{\scriptsize i} & \mbox{\scriptsize\dots} & \mbox{\scriptsize n} \\
\begin{block}{c|cccccc|}
	\mbox{\scriptsize1} & 1 & x_2 & \dots & x_i & \dots & x_n\\
	\mbox{\scriptsize2} & 1 & x_1-\Sigma & \dots & x_i & \dots & x_n\\
	\mbox{\scriptsize\vdots} & \vdots & \vdots & \ddots & \vdots &  & \vdots \\
	\mbox{\scriptsize i} & 1 & x_2 & \dots & x_1-\Sigma & \dots & x_n \\
	\mbox{\scriptsize\vdots} & \vdots & \vdots & & \vdots & \ddots & \vdots \\
	\mbox{\scriptsize n} & 1 & x_2 & \dots & x_i & \dots & x_1-\Sigma \\
\end{block}
\end{blockarray}
\\ &=
\begin{blockarray}{ccccccc}
& \mbox{\scriptsize1} & \mbox{\scriptsize2} & \mbox{\scriptsize\dots} & \mbox{\scriptsize i} & \mbox{\scriptsize\dots} & \mbox{\scriptsize n} \\
\begin{block}{c|cccccc|}
	\mbox{\scriptsize1} & 1 & x_2  & \dots & x_i & \dots & x_n\\
	\mbox{\scriptsize2} & 0 & \delta_2-\Sigma & \dots & 0 & \dots & 0\\
	\mbox{\scriptsize\vdots} & \vdots & \vdots & \ddots & \vdots &  & \vdots \\
	\mbox{\scriptsize i} & 0 & 0 & \dots & \delta_i-\Sigma & \dots & 0 \\
	\mbox{\scriptsize\vdots} & \vdots & \vdots &  & \vdots & \ddots & \vdots \\
	\mbox{\scriptsize n} & 0 & 0 & \dots & 0 & \dots & \delta_n-\Sigma \\
\end{block}
\end{blockarray}
\\ &= \prod_{i=2}^n \left( \delta_i-\Sigma \right).
\end{align*}
Since $\Sigma\ne x_1$ and $x_i\ge0,~\forall i\in\langle n\rangle$, it follows that $\delta_i \ne \Sigma$, $\forall i\in\langle n\rangle\backslash\{1\}$, i.e., $\det(S)\ne0$. Therefore $S$ is invertible.
\end{proof}

%---------------------------------
\begin{theorem} \label{thm:diagonalizable}
If $x \ge 0$, then the nonnegative matrix $P_x$ is diagonalizable.
\end{theorem}

\begin{proof}
Notice that if $\Sigma=x_1$, then $P_x=x_1 I$ is diagonalizable.
Otherwise, if $\Sigma \ne x_1$, then  
\[
PS = 
\begin{bmatrix}
Pe & Pv_2 & \dots & Pv_n
\end{bmatrix}
=
\begin{bmatrix}
\Sigma e & \delta_2v_2 & \dots & \delta_nv_n
\end{bmatrix}
\]
and
\begin{align*}
SD &= 
\begin{bmatrix}
1 & x_2 & \dots & x_i & \dots & x_n\\
1 & x_1 - \Sigma & \dots & x_i & \dots & x_n\\
\vdots & \vdots & \ddots & \vdots & & \vdots \\
1 & x_2 & \dots & x_1 - \Sigma & \dots & x_n \\
\vdots & \vdots & & \vdots & \ddots & \vdots \\
1 & x_2 & \dots & x_i & \dots & x_1 - \Sigma \\
\end{bmatrix}
\begin{bmatrix}
\Sigma & 0 & \dots & 0 & \dots & 0 \\
0 & \delta_2 & \dots & 0 & \dots & 0 \\
\vdots & \vdots & \ddots & \vdots & & \vdots \\
0 & 0 & \dots & \delta_i & \dots & 0 \\
\vdots & \vdots & & \vdots & \ddots & \vdots \\
0 & 0 & \dots & 0 & \dots & \delta_n
\end{bmatrix} \\
&= 
\begin{bmatrix}
\Sigma & \delta_2x_2 & \dots & \delta_ix_i & \dots & \delta_nx_n \\
\Sigma & \delta_2(x_1-\Sigma) & \dots & \delta_ix_i & \dots & \delta_nx_n \\
\vdots & \vdots & \ddots & \vdots & & \vdots \\
\Sigma & \delta_2x_2 & \dots & \delta_i(x_1-\Sigma) & \dots & \delta_nx_n \\
\vdots & \vdots & & \vdots & \ddots & \vdots \\
\Sigma & \delta_2x_2 & \dots & \delta_ix_i & \dots & \delta_n(x_1-\Sigma)
\end{bmatrix} \\
&=
\begin{bmatrix}
\Sigma e & \delta_2v_2 & \dots & \delta_nv_n
\end{bmatrix},
\end{align*}
i.e., $PS = SD$. Since $S$ is invertible, it follows that $P=SDS^{-1}$.
\end{proof}

%---------------------------------
\begin{remark}
Theorem \ref{thm:diagonalizable} demonstrates that symmetry is not essential in the realization of Sule\u{\i}manova spectra. 
\end{remark}

%---------------------------------
\section{New Sufficient Condition}
%---------------------------------

In this section we give necessary and sufficient conditions for a list to be realizable by a permutative matrix of the form \eqref{eq:perm_matrix}. The following result gives a complete characterization and a novel sufficient condition for the DRNIEP (in Section \ref{novsuffsniep}, this condition will be shown to be sufficient for the SNIEP).

%--------------
\begin{theorem} 
\label{thm:problem5}
If \( \Lambda = \{ \lambda_1, \dots, \lambda_n \} \) is a list of real numbers, then \( \Lambda \) is realizable by a matrix of the form \eqref{eq:perm_matrix} if and only if
\begin{equation}\label{tnn}
    s_1(\Lambda) \geq 0                          
\end{equation}
and
\begin{equation} \label{aux}
    s_1 (\Lambda) - n \lambda_i \geq 0,~i\in\langle n\rangle\backslash\{1\}.
\end{equation}
\end{theorem}

\begin{proof}
Immediate in view of \eqref{novsuffcond}.
\end{proof}

If $\mathcal{C}_n := \{ \lambda \in \mathbb{R}^n \mid \lambda~\text{satisfies \eqref{tnn} and \eqref{aux}} \}$, then $\mathcal{C}_n$ is a \emph{polyhedral cone}, and is \emph{finitely generated} (i.e., it has finitely-many \emph{extreme directions}). It is useful to determine the extreme directions of $\mathcal{C}_n$. To this end, let 
\[
    M_n := 
    \begin{cases}
    [1], & n=1 \\
    \begin{bmatrix}
        1 & e^\top \\
        e & -I
    \end{bmatrix}, & n > 1
    \end{cases}
\]
and, with a slight abuse of notation, let 
\begin{equation*}
    \text{span}(M_n) := \{ M_n y \mid y \in \mathbb{R}^n \}
\end{equation*}
and
\begin{align*}
    \coni{M_n} 
    := \{ M_n y \mid y\ge 0 \}                                                                               
    = \left\{ y_1 e + \sum_{i=2}^n y_i(e_1-e_i) \mid y_i \ge 0, ~i \in \langle n \rangle \right\}.
\end{align*}
In \cite{paparella2016} it was shown that 
\[
    M_n^{-1} 
    = \frac{1}{n}\begin{bmatrix}
        1 & e^\top \\
        e & J-nI
    \end{bmatrix},~n>1.
\]

The following result demonstrates that the columns of $M_n$ are the extreme directions of $\mathcal{C}_n$. 

%---------------------------------
\begin{theorem}
    $\mathcal{C}_n = \coni{M_n}$. 
\end{theorem}

\begin{proof}
If $\lambda \in \text{span}(M_n)$, then there is a unique vector $y$ such that $\lambda = M_n y = e y_1 + \sum_{i=2}^{n}y_i(e_1-e_i)$, where $y_i \in \mathbb{R}, ~\forall i\in\langle n\rangle$. In particular, $\lambda_1=\sum_{k=1}^ny_k$ and $\lambda_i=y_1-y_i, ~\forall i\in\langle n\rangle\backslash\{1\}$. Notice that
\begin{equation}
    \sum_{k=1}^n \lambda_k = \sum_{k=1}^n y_k + \sum_{k=2}^n(y_1-y_k) = n y_1,
    \label{eqn:conitnn}
\end{equation}
and
\begin{equation}
    \sum_{k=1}^n \lambda_k - n \lambda_i = n y_1 - n(y_1-y_i) = n y_i,~i=2,\dots,n,
    \label{eqn:coniaux}
\end{equation}
and these inequalities are nonnegative if and only if $y_i \ge 0$, $i=1,\dots,n$, i.e, if and only if $\lambda \in \coni{M_n}$.
\end{proof}

%---------------------------------
\section{Relating Theorem \ref{thm:problem5} to Known Sufficient Conditions} \label{relsuffcond}
%---------------------------------

Mariju\'an et al.~\cite{mps2007} mapped various sufficient conditions for the RNIEP displaying their connections to one another. By finding the location of Theorem \ref{thm:problem5} within their map may provide a measure of importance to the RNIEP. The following are known sufficient conditions for the RNIEP.

%---------------------------------
\begin{theorem}
[Ciarlet \cite{ciarlet1968}]
\label{thm:ciarlet}
If $\Lambda=\{\lambda_1,\dots,\lambda_n\}$ satisfies $\lambda_1\ge|\lambda_i|$ for all $\lambda_i\in\Lambda$ and 
\[
|\lambda_i|\le\frac{\lambda_1}{n-1},~\forall i \in \langle n \rangle \backslash \{1\},
\]
then $\Lambda$ is realizable.
\end{theorem}

%---------------------------------
\begin{theorem}[Sule\u{\i}manova \cite{suleimanova1949}]
\label{thm:suleimanovacond}
If $\Lambda=\{\lambda_1,\dots,\lambda_n\}$ satisfies $\lambda_1\ge|\lambda_i|$ for all $\lambda_i\in\Lambda$ and 
\[
    \lambda_1 + \sum_{\lambda_i<0}\lambda_i \ge 0,
\]
then $\Lambda$ is realizable.
\end{theorem}

\begin{remark}
To differentiate Sule\u{\i}manova spectra from spectra that satisfy the condition in Theorem \ref{thm:suleimanovacond}, we will refer to the condition of Theorem \ref{thm:suleimanovacond} as the \emph{Sule\u{\i}manova condition}.
\end{remark}

\begin{theorem}[Salzmann \cite{s1972}]
\label{thm:salzmann}
If $\lambda_1\ge\lambda_2\ge\dots\ge\lambda_n$ are real numbers such that 
\[
    \sum_{i=1}^n\lambda_i \ge 0
\]
and 
\[
    \lambda_k + \lambda_{n-k+1} \le \frac{2}{n}\sum_{i=1}^n \lambda_i,~\forall k\in\langle\lfloor (n+1)/2\rfloor\rangle\backslash\{1\},
\]
then $\{\lambda_1,\dots,\lambda_n\}$ is realizable by a nonnegative diagonalizable matrix.
\end{theorem}

%---------------------------------
\begin{theorem}[Perfect 1 \cite{perfect1953}]
\label{thm:perfect1}
Let 
\[ \Lambda=\{\lambda_0,\lambda_1,\lambda_{11},\dots,\lambda_{1t_1},\dots,\lambda_r,\lambda_{r1},\dots,\lambda_{rt_r},\delta\}, \] 
where $\lambda_0\ge|\lambda|$ for all $\lambda\in\Lambda$ and 
\[
    \sum_{\lambda\in\Lambda}\lambda\ge 0,~\delta\le 0,
\]
$\lambda_j\ge0$ and $\lambda_{ji}\le 0$ for all $j\in\langle r\rangle$ and $i\in\langle t_j\rangle$. If for every $j\in\langle r\rangle$, it follows that $\lambda_j + \delta \le 0$ and
\[
     \lambda_j + \sum_{i=1}^{t_j}\lambda_{ji}\le0,
\]
then $\Lambda$ is realizable.
\end{theorem}

We note the following relations.

%---------------------------------
\begin{theorem}
\label{thm:relations}
~
\begin{enumerate}
    \item Theorem \ref{thm:problem5} is independent of Ciarlet.
    \item Theorem \ref{thm:problem5} is independent of Perfect 1.
    \item Sule\u{\i}manova spectrum are independent of Ciarlet.
    \item Sule\u{\i}manova spectrum imply Theorem \ref{thm:problem5} and the inclusion is strict.
    \item Theorem \ref{thm:problem5} implies Sule\u{\i}manova's condition and the inclusion is strict.
    \item Theorem \ref{thm:problem5} implies Salzmann and the inclusion is strict.
\end{enumerate}
\end{theorem}
\begin{proof}
~
\begin{enumerate}
    \item Notice that $\{2, 0, -2\}$ satisfies Theorem \ref{thm:problem5} but not Ciarlet and the list $\{2, 1, -1\}$ satisfies Ciarlet but not Theorem \ref{thm:problem5}.
    \item Notice that $\{3, 1, -1\}$ satisfies Theorem \ref{thm:problem5} but not Perfect 1 and the list $\{3, 1, -1, -1\}$ satisfies Perfect 1 but not Theorem \ref{thm:problem5}.
    \item Notice that $\{3, -1, -2\}$ is a Sule\u{\i}manova spectra but doesn't satisfy Ciarlet and the list $\{2, 1, -1\}$ satisfies Ciarlet but is not a Sule\u{\i}manova spectra.
    \item Follows directly from Theorem \ref{thm:paparella2016} and the list $\{ 3, 1, -1 \}$ demonstrates that the inclusion is strict.
    \item Let $\lambda=\begin{bmatrix} \lambda_1 & \cdots & \lambda_n \end{bmatrix}^\top\in\coni{M}$ such that $\lambda_1\ge\dots\ge\lambda_p\ge0>\lambda_{p+1}\ge\dots\ge\lambda_n$. If $p=1$, i.e, if $\lambda$ has only one nonnegative term, then \eqref{tnn} demonstrates that $\lambda$ satisfies Sule\u{\i}manova's condition. Assume that $p\ge 2$. Notice that for all $k\in\langle p\rangle\backslash\{1\}$, 
    \[
        \sum_{i=1}^n\lambda_i - n\lambda_k \ge 0 \Longleftrightarrow \lambda_1 + \sum_{i=p+1}^n\lambda_i \ge n\lambda_k - \sum_{i=2}^p\lambda_i,
    \]
    thus, it suffices to show that $n\lambda_k\ge\sum_{i=2}^p\lambda_i$ for some $k\in\langle p\rangle\backslash\{1\}$. Notice that if $k=2$, then
    \[
        \sum_{i=2}^p\lambda_i \le (p-1)\lambda_2 \le n\lambda_2.
    \]
    Therefore, $\lambda$ satisfies Sule\u{\i}manova's condition and the list $\{2,1,-1\}$ demonstrates that the inclusion is strict. 
    \item Let $\lambda=\begin{bmatrix} \lambda_1 & \cdots & \lambda_n \end{bmatrix}^\top\in\coni{M}$ such that $\lambda_1\ge\dots\ge\lambda_n$. By \eqref{aux} it follows that for all $k\in\langle n\rangle\backslash\{1\}$,
    \[
        \sum_{i=1}^n\lambda_i \ge n\lambda_k.
    \]
    Thus, for all $k\in\langle\lfloor(n+1)/2\rfloor\rangle\backslash\{1\}$,
    \[
        2\sum_{i=1}^n\lambda_i \ge n\lambda_k + n\lambda_{n-k+1}.
    \]
    Therefore, $\lambda$ satisfies Salzmann and the list $\{ 3, 1, -2, -2 \}$ demonstrates that the inclusion is strict.\qedhere
\end{enumerate}
\end{proof}

Figure \ref{fig:map} contains an updated map as constructed by Mariju\'an et al.~\cite{mps2007}.

\begin{figure}[H]
\centering
\begin{tikzcd}[arrows=Rightarrow, row sep=normal, column sep=0.1em]
\text{Sule\u{\i}manova spectrum} \arrow[d] & \text{Ciarlet} \arrow[d] \\
\text{Theorem \ref{thm:problem5}} \arrow[r] \arrow[d] & \text{Sule\u{\i}manova's condition} \arrow[d] \arrow[r] & \text{Sule\u{\i}manova-Perfect} \arrow[d] \\
\text{Salzmann} \arrow[r] & \text{Fiedler} \Longleftrightarrow \text{Soto 1} \arrow[d] \Longrightarrow \text{Kellogg} \arrow[r] & \text{Borobia} \\
& \text{Soto 2} & \text{Perfect 1} \arrow[u]
\end{tikzcd}
\caption{Relating Theorem 4.1 to known sufficient conditions.} \label{fig:map}
\end{figure}
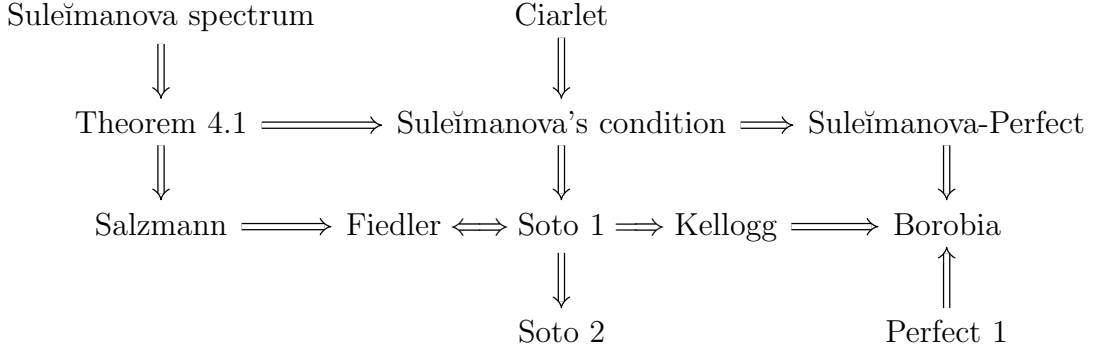

%---------------------------------
\section{Theorem \ref{thm:problem5} and the SNIEP} \label{novsuffsniep}
%---------------------------------

In this section, it is shown that any list satisfying \eqref{tnn} and \eqref{aux} is SNIEP realizable. 

The following result is due to Fiedler \cite{fiedler1974}.

%---------------------------------
\begin{lemma}
[{\cite[Lemma 2.2]{fiedler1974}}] 
\label{lemma:Fiedler}
Let $A$ be a symmetric $m$-by-$m$ matrix with eigenvalues $\alpha_1,\dots,\alpha_m$, let $u$, $\|u\|=1$, be a unit eigenvector corresponding to $\alpha_1$; let $B$ be a symmetric $n$-by-$n$ matrix with eigenvalues $\beta_1,\dots,\beta_n$, let $v$, $\|v\|=1$, be a unit eigenvector corresponding to $\beta_1$.

Then for any $\rho$, the matrix
\[
C = \begin{bmatrix}
A & \rho u v^\top \\
\rho v u^\top & B
\end{bmatrix}
\]
has eigenvalues $\alpha_2,\dots,\alpha_m,\beta_2,\dots,\beta_n,\gamma_1,\gamma_2$, where $\gamma_1,\gamma_2$ are eigenvalues of 
\begin{equation}
\label{cmatrix}
    \hat{C} = \begin{bmatrix}
    \alpha_1 & \rho \\
    \rho & \beta_1
    \end{bmatrix}.
\end{equation}
\end{lemma}

Fiedler uses this result to show the following. 

%---------------------------------
\begin{theorem}
[{\cite[Theorem 2.3]{fiedler1974}}]
\label{thm:fiedler}
Let $A \in M_m(\mathbb{R})$ and $B \in M_n(\mathbb{R})$ be nonnegative, symmetric matrices. Suppose that $\sigma(A) = \{ \alpha_1,\dots,\alpha_m \}$, $\sigma(B) = \{ \beta_1,\dots,\beta_n \}$, and $\rho(A) = \alpha_1 \ge \beta_1 = \rho(B)$. Further suppose that $u$ and $v$ are the normalized, nonnegative right eigenvectors corresponding to $\alpha_1$ and $\beta_1$, respectively. If $\theta \ge 0$, $\rho:=(\theta(\alpha_1-\beta_1+\theta))^{1/2}$, and $C$ is the matrix defined as in \eqref{cmatrix}, then $C$ is nonnegative and $\sigma(C) = \{\alpha_1+\theta,\alpha_2,\dots,\alpha_m,\beta_1-\theta,\beta_2,\dots,\beta_n \}$.
\end{theorem}

%---------------------------------
\begin{theorem}
\label{thm:korders}
If $\lambda\in\coni{M_n}$ and $\Lambda = \{ \Lambda_1, \dots, \lambda_n \}$, then $\Lambda$ is symmetrically realizable.
\end{theorem}
\begin{proof}
Proceed by induction on $n$. The result is trivial for the base-case when $n=1$; if $n=2$ and $\lambda = M_2 y$, $y \ge 0$, then
\[ \lambda = 
\begin{bmatrix}
\lambda_1 \\
\lambda_2
\end{bmatrix}
= 
\begin{bmatrix}
1 & 1 \\
1 & -1 
\end{bmatrix}
\begin{bmatrix}
y_1 \\
y_2
\end{bmatrix}
= 
\begin{bmatrix}
y_1 + y_2 \\
y_1 - y_2
\end{bmatrix}
\]
and the nonnegative matrix
\[ 
\begin{bmatrix}
y_1 & y_2 \\
y_2 & y_1
\end{bmatrix}
\]
has eigenvalues $\lambda_1 = y_1 + y_2$ and $\lambda_2 = y_1 - y_2$.

Assume that the result holds for all lists of length $k$, with $k \ge 2$. If $\lambda \in \coni{M_{k+1}}$, then there are nonnegative scalars $y_1,\dots,y_{k+1}$ such that
\[ 
\lambda =
\begin{bmatrix}
\lambda_1   \\
\lambda_2   \\
\vdots      \\
\lambda_{k+1}
\end{bmatrix}
=
\begin{bmatrix}
\sum_{i=1}^{k+1} y_i  \\
y_1 - y_2   \\
\vdots      \\
y_1 - y_{k+1}
\end{bmatrix}.
\]
If 
\[
\alpha :=
\begin{bmatrix}
\sum_{i=1}^{k} y_i  \\
y_1 - y_2   \\
\vdots      \\
y_1 - y_{k}
\end{bmatrix}~\text{and}~
\beta:=
\begin{bmatrix}
y_1
\end{bmatrix},
\]
then $\alpha \in \coni{M_k}$ and $\beta \in \coni{M_1}$. The result follows by applying Theorem \ref{thm:fiedler} with $\theta = y_{k+1}$.
\end{proof}

The proof of Theorem \ref{thm:korders} is similar to the proof that Fiedler gives in showing that every Sule\u{\i}manova spectrum is symmetrically realizable. Detracting from both demonstrations is that they are not constructive, i.e., a realizing matrix is not explicitly given. Johnson and Paparella \cite{johnsonpaparella2016} give a constructive proof that every Suleimanova spectrum is symmetrically realizable via similarity by Hadamard matrices, which we will show gives a constructive proof for lists satisfying \eqref{tnn} and \eqref{aux}.  

%---------------------------------
\begin{definition}
[Hadamard matrix]
An $n$-by-$n$ matrix $H=[h_{ij}]$ is called a \textit{Hadamard matrix} if $h_{ij}\in\{-1,1\}$ and $HH^\top=nI$. 
\end{definition}
A positive integer \(n\) is called a \emph{Hadamard order} if there is a Hadamard matrix of order \(n\). It is well-known, and otherwise easy to establish, that if \(n\) is a Hadamard order greater than one, then \(n\) can not be odd. It is also well-known that every power of two is a Hadamard order. In particular, Sylvester \cite{sylvester1867} showed that if
\[
    W_{2^n} := 
    \begin{cases}
    [1], & n = 0 \\
    \begin{bmatrix}
        W_{2^{n-1}} & \phantom{+} W_{2^{n-1}} \\
        W_{2^{n-1}} & -W_{2^{n-1}}
    \end{bmatrix}, & n > 0
    \end{cases}
\]
then \( W_{2^n} \) is a Hadamard matrix of order \(2^n\). The matrix $W_{2^n}$ is also known as the \textit{Walsh matrix or order $2^n$}.

When the first row and column of a Hadamard matrix $H$ is positive, we say that $H$ is a \textit{normalized Hadamard matrix}. Furthermore, any Hadamard matrix can be normalized to obtain another Hadamard matrix by scaling the rows and columns as needed. As such, and without loss of generality, we may assume that any Hadamard matrix is a normalized Hadamard matrix.

A real list $\Lambda =\{ \lambda_1, \dots, \lambda_n \}$ is called \emph{normalized} if $\lambda_1 = 1 \ge \lambda_2 \ge \cdots \ge \lambda_n$. Johnson and Paparella \cite{johnsonpaparella2016} used Hadamard matrices to establish the following result (recall that a nonnegative matrix $A$ is called \emph{doubly stochastic} if every row and column sum to unity).

%---------------------------------
\begin{theorem}
[{\cite[Theorem 6.3]{johnsonpaparella2016}}]
\label{thm:jpap16_hdh}
If $\Lambda=\{\lambda_1,\dots,\lambda_n\}$ is a normalized Sule\u{\i}ma\-nova spectrum, then $\Lambda$ is realizable by a symmetric, doubly stochastic matrix.
\end{theorem}

% \begin{remark}
% By relaxing the restrictions on $\sigma$ in Theorem \ref{thm:jpap16_hdh} to allow for non-normalized spectra, we still obtain a realizing matrix, but cannot guarantee it will be doubly stochastic.
% \end{remark}

In particular, it is shown that if $D=\operatorname{diag}(\lambda_1, \dots, \lambda_n)$, $H$ is a normalized Hadamard matrix of order $n$, and $A = n^{-1}H D H^\top$, then $A$ is a symmetric, doubly stochastic matrix with spectrum $\Lambda$. We show that this result extends to normalized lists satisfying \eqref{tnn} and \eqref{aux}. 

In what follows, $\conv{M_n}$ denotes the convex hull of the columns of $M_n$. 

%---------------------------------
\begin{theorem}
\label{thm:had_coni}
Let $\lambda \in \conv{M}$, $D=\operatorname{diag}(\lambda)$, $\Lambda = \{\lambda_1,\dots,\lambda_n \}$, and $H = H_m$ be a normalized Hadamard matrix of order $n$. Then the matrix $n^{-1}HDH^\top$ is a symmetric, doubly stochastic matrix with spectrum $\Lambda$.
\end{theorem}
\begin{proof}
For every $k\in\langle n\rangle$, let $D_k := \lambda_k e_ke_k^\top$ and let 
\begin{align*}
    M_k := HD_kH^\top &= \lambda_kHe_k(He_k)^\top.
\end{align*}
Notice that $D=\sum_{k=1}^n D_k$. 

Because the entries of $He_k(He_k)^\top$ are in $\{-1,1\}$, it follows that the entries of $J-He_k(He_k)^\top$ are in $\{0,2\}$. Furthermore, $\sum_{k=1}^nHe_k(He_k)^\top = HIH^\top = nI$. 

Since $\lambda\in\conv{M}$, it follows that there are nonnegative scalars $y_1,\dots,y_n$ summing to unity such that $\lambda_1 = \sum_{k=1}^n y_k = 1$ and $\lambda_i=y_1-y_i$ for $i\in\langle n\rangle\backslash\{1\}$. Notice that $\Lambda$ is normalized and 
\begin{align*}
        H D H^\top 
        &= \sum_{k=1}^n M_k = \lambda_1He_1(He_1)^\top + \sum_{k=2}^n\lambda_kHe_k(He_k)^\top \\
        &= \left(\sum_{k=1}^n y_k \right) He_1(He_1)^\top + \sum_{k=2}^n (y_1-y_k)He_k(He_k)^\top \\
        %&= y_1He_1(He_1)^\top+\sum_{k=2}^n y_kHe_1(He_1)^\top + y_1\sum_{k=2}^nHe_k(He_k)^\top - \sum_{k=2}^n y_kHe_k(He_k)^\top \\
        &= \sum_{k=1}^n y_1\left(He_k(He_k)^\top\right) + \sum_{k=2}^n y_k\left(He_1(He_1)^\top-He_k(He_k)^\top\right) \\
        &= y_1nI + \sum_{k=2}^n y_k\left(J-He_k(He_k)^\top\right) \ge 0.
\end{align*}
The matrix $n^{-1}H DH^\top$ is clearly symmetric, has spectrum $\Lambda$, and is doubly stochastic $e$ is a right eigenvector corresponding to the Perron root $1$.
\end{proof}

\begin{remark}
Notice that $(\lambda_i,~n^{-1/2}He_i)$ is a right-eigenpair of $n^{-1}HDH^\top$, $\forall i\in\langle n\rangle$. In particular, $(\lambda_1,~n^{-1/2}e)$ is a right-eigenpair.
\end{remark}

The following lemma will be useful for our next main result.

%---------------------------------
\begin{lemma}
\label{lemma:coni_split}
Let $\lambda\in\coni{M_{m+n}}$, $m,n \ge 1$. If 
\[ \alpha := \begin{bmatrix}\lambda_1-\theta \\ \lambda_2 \\ \vdots \\ \lambda_m \end{bmatrix}~\text{and}~\beta := \begin{bmatrix} \lambda_{m+1} + \theta \\ \lambda_{m+2} \\ \vdots \\ \lambda_{m+n} \end{bmatrix}, \] with
\begin{equation}
\label{eqn:sigma}
    \theta := \frac{n}{m+n}\sum_{i=1}^{m+n}\lambda_i - \sum_{i=m+1}^{m+n}\lambda_i,
\end{equation}
then $\alpha\in\coni{M_m}$ and $\beta\in\coni{M_n}$.
\end{lemma}

\begin{proof}
If $\lambda\in\coni{M_{m+n}}$, then $\lambda = y_1e + \sum_{i=2}^{m+n} y_i(e_1-e_i)$, where $y_i \ge 0,~\forall i\in\langle m+n\rangle$, i.e., $\lambda_1=\sum_{i=1}^{m+n}y_i$ and $\lambda_i=y_1-y_i,~ \forall i \in \langle m+n\rangle \backslash\{1\}$. Notice that
\begin{align*}
    \theta 
    &= \frac{n}{m+n} \left( \sum_{i=1}^{m+n} y_i + \sum_{i=2}^{m+n} (y_1-y_i) \right) - \sum_{i=m+1}^{m+n} (y_1-y_i) \\
    &= ny_1-\sum_{i=m+1}^{m+n} (y_1-y_i) = \sum_{i=m+1}^{m+n} y_i \ge 0.
\end{align*}
Thus,
\begin{equation*}
    \alpha 
    = 
    \begin{bmatrix}
    \lambda_1 - \theta  \\
    \lambda_2           \\
    \vdots              \\
    \lambda_m 
    \end{bmatrix}
    = \begin{bmatrix}
    \sum_{i=1}^m y_i    \\
    y_1 - y_2           \\
    \vdots              \\
    y_1 - y_m 
    \end{bmatrix}                                                                               
    = y_1e + \sum_{i=2}^m y_i(e_1-e_i)
\end{equation*}
and
\begin{equation*}
    \beta 
    = 
    \begin{bmatrix}
    \lambda_{m+1} + \theta  \\
    \lambda_{m+2}           \\
    \vdots                  \\
    \lambda_{m+n} 
    \end{bmatrix}
    = 
    \begin{bmatrix}
    y_1 + \sum_{i=m+2}^{m+n} y_i    \\
    y_1 - y_{m+2}                   \\
    \vdots                          \\
    y_1 - y_{m+n} 
    \end{bmatrix}                                                                               
    = y_1e + \sum_{i=2}^n y_{m+i}(e_1-e_i),
\end{equation*} 
i.e., $\alpha \in \coni{M_m}$ and $\beta \in \coni{M_n}$
\end{proof}

%---------------------------------
\begin{observation}
\label{obs:rho}
If $\alpha$ and $\beta$ are defined as above, then $\rho := (\theta(\alpha_1-\beta_1+\theta))^{1/2}$ is real and nonnegative.
\end{observation}

\begin{proof}
Since $\alpha_1 = \sum_{i=1}^m y_i$ and $\beta_1 = y_1 + \sum_{i=m+2}^{m+n} y_i$, it follows that 
\[
    \alpha_1-\beta_1+\theta = \left(\sum_{i=1}^m y_i - y_1 - \sum_{i=m+2}^{m+n} y_i + \sum_{i=m+1}^{m+n} y_i\right) = \sum_{i=2}^{m+1} y_i \ge 0.
\]
Therefore, $\rho$ is real and nonnegative.
\end{proof}

%---------------------------------
\begin{corollary}
\label{thm:2orders}
Let \( \lambda\in \coni{M_{m+n}}\), \(\alpha=(\lambda_1-\theta,\lambda_2,\dots,\lambda_m)\), and \(\beta=(\lambda_{m+1}+\theta, \lambda_{m+2},\dots,\lambda_{m+n})\), where $m$ and $n$ are Hadamard orders and $\theta$ is given by \eqref{eqn:sigma}. Then the symmetric nonnegative matrix
\[
    C = \begin{bmatrix}
    A & \rho J/\sqrt{mn} \\
    \rho J^\top/\sqrt{mn} & B
    \end{bmatrix},
\]
where $A=m^{-1}H_m\operatorname{diag}(\alpha)H_m^\top$, $B=n^{-1}H_n\operatorname{diag}(\beta)H_n^\top$, and $\rho=(\theta(\lambda_1-\lambda_{m+1}-\theta)/mn)^{1/2}$, has spectrum $\Lambda$.
\end{corollary}

\begin{proof}
The result follows by a straightforward application of Lemma \ref{lemma:Fiedler}, Theorem \ref{thm:had_coni}, Lemma \ref{lemma:coni_split}, and Observation \ref{obs:rho}.
\end{proof}

%---------------------------------
\begin{example}
Let $\lambda=\begin{bmatrix} 20 & 1 & -2 & -3 & -4 & -5 \end{bmatrix}^\top$. Notice that $\lambda \in \coni{M_6}$ since 
\[ \lambda = M_6 y, \]
with \( y = 1/6 \begin{bmatrix} 7 & 1 & 19 & 25 & 31 & 37 \end{bmatrix}^\top\). 

With $m=4$, $n=2$, $\theta=34/3$, $\rho = \sqrt{323/18}$, $\alpha=(26/3,1,-2,-3)$, and $\beta=(22/3,-5)$, it follows that the matrix
\[
    C = 
    \frac{1}{6}
    \left[
    \begin{array}{cccc|cc}
        7 & 13 & 22 & 10 & \sqrt{646} & \sqrt{646} \\
        13 & 7 & 10 & 22 & \sqrt{646} & \sqrt{646} \\
        22 & 10 & 7 & 13 & \sqrt{646} & \sqrt{646} \\
        10 & 22 & 13 & 7 & \sqrt{646} & \sqrt{646} \\
        \hline 
        \sqrt{646} & \sqrt{646} & \sqrt{646} & \sqrt{646} & 7 & 37 \\
        \sqrt{646} & \sqrt{646} & \sqrt{646} & \sqrt{646} & 37 & 7 \\
        
    \end{array}
    \right],
\] 
has $\Lambda$ as its spectrum.
\end{example}

We conclude by noting that the conditions in Section \ref{relsuffcond} are sufficient for SNIEP realizability \cite{mps2017}; thus, the map contained in Figure \ref{fig:map} applies to the SNIEP.

\section*{Acknowledgment}

We gratefully acknowledge the comments and suggestions from the anonymous referee and from handling-editor Raphael Loewy. 

\bibliography{references}
\bibliographystyle{abbrv}

\end{document}